\documentclass{amsart}
\usepackage{a4}
\usepackage{amssymb}
\def\FF{{\mathbb F}}
\def\squareforqed{\hbox{\rlap{$\sqcap$}$\sqcup$}}
\def\qed{\ifmmode\squareforqed\else{\unskip\nobreak\hfil
\penalty50\hskip1em\null\nobreak\hfil\squareforqed
\parfillskip=0pt\finalhyphendemerits=0\endgraf}\fi\medskip}

\newcommand{\diag}{\mathrm{diag}}

\newcommand{\SL}{\mathrm{SL}}
\newcommand{\SO}{\mathrm{SO}}
\newcommand{\ombar}{\bar{\omega}}
\newcommand{\POmega}{\mathrm{P\Omega}}
\newcommand{\Spin}{\mathrm{Spin}}
\newcommand{\Tr}{\mathrm{Tr}}
\newtheorem{theorem}{Theorem}
\newtheorem{lemma}{Lemma}
\newtheorem{corollary}{Corollary}

\newtheorem{definition}{Definition}
\title
[Albert algebras and construction of $F_4$ and $E_6$]
{Albert algebras and construction of the finite simple groups $F_4(q)$, $E_6(q)$ 
and ${}^2E_6(q)$ and their generic covers}
\author{Robert A. Wilson}
\address{School of Mathematical Sciences, Queen Mary University of London,
Mile End Road, London E1 4NS, UK}
\email{R.A.Wilson@qmul.ac.uk}
\date{First draft 05/08/10; this version 18/10/13}

\begin{document}
\maketitle

\begin{abstract}
We give a uniform construction of the finite simple groups $E_6(q)$, $F_4(q)$ and
${}^2E_6(q)$, which does not require any special treatment for characteristics
$2$ or $3$, and in particular avoids any mention of quadratic Jordan algebras.
Although almost all the ingredients can already be found scattered through
research papers spanning more than a century, a coherent, sef-contained,
account is hard to find in the literature.
%
\end{abstract}
\tableofcontents
\section{Introduction}
The construction of the finite simple groups $E_6(q)$ and their triple covers
(which exist whenever $q\equiv 1\bmod 3$) goes back over 100 years
to the work of Dickson \cite{Dickson1,Dickson2}. This work has been, perhaps
unjustly, somewhat neglected following Chevalley's uniform construction
 in 1955 of what are now called
Chevalley groups \cite{Chevalley}, which include five of the ten families of exceptional groups
of Lie type, in particular $E_6(q)$. This is in spite of the fact that \cite{Chevalley}
constructs only the simple groups, and not their generic covers.
Moreover, the representation is on the Lie algebra, which has dimension $78$,
as opposed to the smallest representation, which has dimension $27$.

The other major breakthrough since Dickson is the 
discovery of the exceptional Jordan algebra (or Albert algebra),
 which (in the real case)
was discovered by physicists in the 1930s
as a by-product of an unsuccessful attempt to find an algebraic underpinning for quantum mechanics
\cite{JNW}.  This $27$-dimensional 
algebra consists of $3\times 3$ Hermitian matrices over
Cayley numbers, with multiplication $X\circ Y=\frac12(XY+YX)$. 
Freudenthal \cite{Freud} showed that $E_6$ is the stabiliser of the
`determinant',
a certain cubic form defined on this space.
Seligman showed that the automorphism group of a split Jordan algebra over
any field $F$ is isomorphic to the Chevalley group $F_4(F)$.
Jacobson \cite{Jac1,Jac2,Jac3} studied this construction of $F_4$ in detail
and generalized the construction of $E_6$ to
arbitrary fields of characteristic not $2$ or $3$. By this stage it must have been 
implicit that the determinant is essentially the same as Dickson's
cubic form, although Jacobson does not refer
to Dickson, and I have not found an explicit identification in the literature
earlier than \cite{Magaard}.
Moreover, fields of characteristic $2$ and $3$ are still problematic in the
Jordan algebra context, although they were no obstacle to Dickson.
 
Chevalley and Schafer \cite{CheSch} showed that the algebra of derivations
of the real Albert algebra is a Lie algebra of type $F_4$, and also showed how to
extend this to $E_6$ by adjoining right-multiplications by matrices with trace $0$.
Corresponding descriptions of the groups of automorphisms, generated by
maps $X\mapsto \overline{M}^\top XM$ for certain $3\times 3$ matrices $M$ over
complex subfields of the Cayley numbers, are given by Jacobson \cite{Jac3},
who attributes them to Freudenthal, in the revised Russian translation of \cite{Freud}.
See also \cite{DM}, and the 1985 reprint of \cite{Freud}.

It was only in the late 1980s, when the maximal subgroup problem came to
prominence, that there was renewed interest in Dickson's work.
Of particular note are Magaard's unpublished thesis \cite{Magaard} on maximal
subgroups of $F_4(q)$ in characteristic at least $5$, and 
the series of papers by
Aschbacher \cite{Asch1,Asch2,Asch3,Asch4,Asch5} on maximal subgroups
of $E_6(q)$. In these papers, the $27$-dimensional representation of the generic cover
reveals much more structure than the $78$-dimensional representation
on the Lie algebra, and leads to strong restrictions on the shape of a
maximal subgroup. However, the fact that Aschbacher apparently decided not
to attempt to get a complete list of maximal subgroups means that there is still
a need for a modern version of Dickson's construction, to provide a starting point
for investigation of this and other problems.
It is our aim in this paper to develop this theory in a characteristic-free way,
and in particular to remove the restriction to characteristic not $2$ or $3$. 
The main achievement is a relatively straightforward derivation
of the group order, which is a notoriously difficult problem from the
Lie-theoretic point of view. 

\section{The real exceptional Jordan algebra}
First we recall the definition and basic properties
of the real Albert algebra (or exceptional Jordan algebra), 
$\mathbb J=\mathbb J_{\mathbb R}$.
It consists of $3\times 3$ Hermitian matrices over 
the Cayley numbers (also known as octonions). 
We write
\begin{eqnarray}
(a,b,c\mid A,B,C) &=& \begin{pmatrix}a&C&\overline{B}\cr \overline{C}&b&A\cr
B&\overline{A}&c\end{pmatrix}.
\end{eqnarray}
The Jordan product $X\circ Y$ of two such matrices is $\frac12(XY+YX)$, in terms
of the ordinary matrix product $XY$. It can be readily checked that the algebra
is closed under this multiplication. Moreover, $X\circ X=XX$, with the ordinary
matrix product, so we shall write $X^2=X\circ X$. Also, by commutativity we have
$$(X\circ X)\circ X=X\circ(X\circ X),$$ so we write $X^3=X\circ X\circ X$
(but note that we cannot write this as $XXX$, since it is not necessarily
the case that $X(XX)=(XX)X$). 

Now by explicit computation we can verify that any matrix $X=(a,b,c\mid A,B,C)$ in the exceptional
Jordan algebra satisfies a form of the Cayley--Hamilton Theorem, specifically
\begin{eqnarray}
X^3
&=& \Tr(X).X^2 
+ Q(X) 
.X + \det(X).I
\end{eqnarray}
where
the 
determinant $\det$ and the quadratic form $Q$ are defined by
\begin{eqnarray}
 Q(X) &=& \frac12(\Tr(X^2
)-\Tr(X)^2)\cr
&=&A\overline{A}+B\overline{B}+C\overline{C}-ab-ac-bc\cr
\det
(X) &=& abc-aA\overline{A}-bB\overline{B}-cC\overline{C}
+ (AB)C+\overline{C}(\overline{B}.\overline{A}).
\end{eqnarray}
It follows that any automorphism of the algebra preserves the trace
$\Tr(X)$, the standard norm $N(X)=\Tr(X^2)$, and the determinant.
Moreover, taking traces in the Cayley--Hamilton Theorem and re-arranging
gives
\begin{eqnarray}
\det(X) &=& \frac13\Tr(X^3
) -\frac12\Tr(X^2
)\Tr(X) + \frac16\Tr(X)^3.
\end{eqnarray}

Conversely, if the trace, the norm and the determinant are all preserved, then 
the multiplication can be recovered as follows.
Polarizing the norm by
\begin{eqnarray}
2b(X,Y) &=& N(X+Y)-N(X)-N(Y)\cr
&=& 2\Tr(X\circ Y)
\end{eqnarray}
gives an inner product $b$.
Similarly, 
we may polarize the cubic form $\Tr(X^3)$ to obtain a symmetric trilinear form $t$
given by
\begin{eqnarray}
24t(X,Y,Z)&=& \Tr((X+Y+Z)^3)+\Tr((X-Y-Z)^3)\cr
&&\qquad {} +\Tr((Y-X-Z)^3)+\Tr((Z-X-Y)^3)
\end{eqnarray}
Now by explicit computation it can be checked that
\begin{eqnarray}
\Tr((X\circ Y)\circ Z) &=& \Tr (X\circ (Y\circ Z))
\end{eqnarray}
and it then follows that
\begin{eqnarray}
t(X,Y,Z)&=&\Tr((X\circ Y)\circ Z) = b(X\circ Y,Z).
\end{eqnarray}
Since the norm $N$ is positive-definite, $b$ is non-singular, 
and therefore knowledge of the inner products
$b(X\circ Y,Z)$ as $Z$ runs over a basis determines $X\circ Y$ uniquely.

 The (compact real form of the) Lie group $F_4$ 
may be defined as the automorphism group
of the algebra $\mathbb J$, although of course this was not the original definition.
Similarly, 
a particular
group of type $E_6$ 
is the group of linear maps which preserve the determinant.
(This real form of $E_6$ is neither split nor compact.)
We may alternatively define $F_4$ as the stabilizer of the identity matrix
in $E_6$, since if the determinant is preserved, and the identity matrix is fixed, then
the trace of $X$ is $t(I,I,X)$,
and the norm of $X$ is $t(I,X,X)$, so these are also preserved.

If $M$ is any $3\times 3$ matrix written over
(any) complex subalgebra of the (real) octonions, 
then the operation $X\mapsto \overline{M}^\top X M$
makes sense, 
because
each entry in $\overline{M}^\top X M$ is a sum of terms of the form $m_1xm_2$, where
$m_1$ and $m_2$ lie in this copy of the complex numbers, and so $m_1(xm_2)=(m_1x)m_2$.
It is clear by restricting to complex matrices $X$ that such an operation can only
preserve the determinant if $|\det M|=1$.
Conversely, 
we use the fact that any complex matrix of determinant $\pm1$ is 
(plus or minus) a product of fundamental transvections, and check explicitly
that the fundamental  transvections preserve the determinant
(see Lemma~\ref{tvpresdet}). Thus any complex matrix of determinant $\pm 1$
preserves the determinant. On the other hand, if $u\overline{u}=1$ but $u\ne\pm 1$,
then it is easy to produce examples to show that $\diag(u,1,1)$ does not preserve the
determinant. Hence the same is true for any matrix of determinant $u$.
Therefore a complex matrix $M$ preserves the determinant if and only if
$\det M=\pm 1$. Negating $M$ if necessary, we may assume $\det M=1$.

Finally, in order for $M$ to preserve the
identity element of the algebra, and hence to lie in $F_4$,
it is necessary and sufficient to have the extra condition $\overline{M}^\top M = I$.
It is shown in \cite{DM} that 
the compact real form of $F_4$ is generated by such elements.

\section
{Split octonions and the Dickson--Freudenthal determinant}
Much the  same constructions  work over finite fields, 
except that there are obvious difficulties 
in characteristics $2$ and 
$3$ caused by dividing by $2$ or $3$.
To overcome these difficulties we have to be careful to choose the most useful form
of each definition from the various no-longer-equivalent versions.

For example, the usual `compact' version of the octonions does not work in characteristic $2$,
so we use instead the `split' version, which works over any field.
See for example \cite[Section 4.3.3]{FSG} for the equivalence of the two versions
over finite fields of odd characteristic.
\begin{definition}
If $F$ is any field, the \emph{split octonion algebra} over $F$ is an $8$-dimensional
vector space $\mathbb O=\mathbb O_F$ over $F$, with basis
$\{e_i\mid i\in \pm I\}$, where $I =\{0,1,\omega,\ombar\}$ and $\pm I =
\{{\pm0},{\pm 1},{\pm\omega},{\pm\ombar}\}$, and bilinear
multiplication given by
\begin{enumerate}
\item $e_1e_{\omega}=-e_{\omega}e_1=e_{-\ombar}$;
\item $e_1e_0=e_{-0}e_1=e_1$;
\item $e_{-1}e_1=-e_0$ and $e_0e_0=e_0$;
\end{enumerate}
and images under negating all suffices (including $0$), and multiplying all
suffices by $\omega$, where $\omega^2=\ombar$ and $\omega\ombar=1$.
All other products of basis vectors are $0$.
\end{definition}
Thus $e_{\pm0}$ are orthogonal idempotents, and $e_0+e_{-0}=1$.
This is essentially the same definition as given in (4.37) of \cite{FSG}, 
but with the basis vectors
$x_1,\ldots, x_8$ in \cite{FSG} corresponding respectively to
$e_{-1}$, $e_{\ombar}$, $e_{\omega}$, $e_0$, $e_{-0}$, $e_{-\omega}$,
$e_{-\ombar}$, $e_1$.
In this form of the octonions it no longer makes sense to talk about the `real part'
of $\sum_{i\in \pm I} \lambda_ie_i$ 
and we define instead the \emph{trace} by
\begin{eqnarray}
\Tr(\sum_{i\in \pm I} \lambda_ie_i)=\lambda_0+\lambda_{-0}.
\end{eqnarray} 
Similarly the anti-automorphism $x\mapsto \overline{x}$ of the octonions now
takes the form
$$e_0\leftrightarrow e_{-0}, e_i\mapsto -e_i (i\ne\pm0).$$
Note that this anti-automorphism reverses the order of multiplication,
in the sense that $\overline{xy}=\overline{y}.\overline{x}$, as is easily checked
directly from the definition. Moreover, we see that $\Tr(x)=x+\overline{x}$.
It is easy to compute the norm $N(x)=x\overline{x}$ of an arbitrary element
to be
$$N(\sum_{i\in \pm I}\lambda_ie_i) = \sum_{i\in I
} \lambda_i\lambda_{-i}.$$
This  norm  can be polarized to obtain an inner product $B$
by
$$B(x,y)=N(x+y)-N(x)-N(y).$$

It is easy to see that $\mathbb O_F$ is non-commutative and non-associative,
so that in general $x(yz)\ne(xy)z$.  However, we do have the following.
\begin{lemma}
If $x,y,z\in\mathbb O_F$, then $\Tr(x(yz))=\Tr((xy)z)$.
\end{lemma}
\begin{proof}
Since both sides are trilinear, it suffices to check on a basis. 
If $i+j+k\ne\pm 0$, then $$\Tr(e_i(e_je_k))=0=\Tr((e_ie_j)e_k).$$ Otherwise 
we show that in fact $e_i(e_je_k)=(e_ie_j)e_k$, as follows.
Using the symmetry we find there are just $8$ cases to check, of which the following are
a representative sample:
\begin{eqnarray}
e_0(e_1e_{-1})=-e_0e_{-0}=&0&=(e_0e_1)e_{-1}\cr
e_0(e_{-1}e_1)=-e_0e_0=&-e_0&=e_{-1}e_1=(e_0e_{-1})e_1\cr
e_1(e_0e_{-1})=&e_1e_{-1}&=(e_1e_0)e_{-1}\cr
e_1(e_{\omega}e_{\ombar})=e_1e_{-1}=&-e_{-0}&=e_{-\ombar}e_{\ombar}=
(e_1e_{\omega})e_{\ombar}
\end{eqnarray}
\end{proof}
Since $\Tr(xy)=\Tr(yx)$, it follows that $\Tr(xyz)$ is independent of bracketing, and
cyclic permutations of $x,y,z$. However, in general we have
$$\Tr(xyz)\ne\Tr(xzy).$$

It is also worth noting that the norm is multiplicative.
\begin{lemma}
If $x,y\in\mathbb O_F$, then $N(xy)=N(x)N(y)$.
\end{lemma}
\begin{proof}
We multiply the basis vectors on the left by an arbitrary element
of $\mathbb O_F$, say $$x=
\sum_{i\in \pm I} \lambda_ie_i,$$ 
and obtain
\begin{eqnarray}
xe_0&=&\lambda_0e_0+\lambda_1e_1+\lambda_\omega e_{\omega} + \lambda_{\ombar}
e_{\ombar}\cr
xe_1&=&-\lambda_{-1}e_0+\lambda_{-0}e_1+
\lambda_{\ombar} e_{-\omega} - \lambda_{\omega}
e_{-\ombar}\cr
xe_{\omega}&=&-\lambda_{-\omega}e_0-\lambda_{\ombar}e_{-1}+
\lambda_{-0}e_{\omega} + \lambda_{1}
e_{-\ombar}\cr
xe_{\ombar}&=&-\lambda_{-\ombar}e_0+\lambda_\omega e_{-1}
-\lambda_1 e_{-\omega} + \lambda_{-0}
e_{\ombar}
\end{eqnarray}
and the corresponding equations with all subscripts negated,
from which it is easy to see that $N(xe_i)=0$, and
the inner products of distinct basis vectors are
all multiplied by $N(x)=\sum_{i\in I
} \lambda_i\lambda_{-i}$.
Hence the result follows by linearity.
\end{proof}
The final basic property of the split octonions 
is the
\emph{Moufang law}
which comes in three equivalent versions.
\begin{lemma}
For all $x,y,z,\in\mathbb O$, the following identities hold:
\begin{eqnarray}
x(yz)x&=&(xy)(zx),\cr
x(yzy)&=&((xy)z)y,\cr
(xyx)z&=&x(y(xz)).
\end{eqnarray}
\end{lemma}
\begin{proof}
As we shall not need these identities in the rest of the paper, 
we merely sketch the proof of the first one. 
Let $x=\sum_{i\in \pm I} \lambda_i e_i$. By bilinearity we need only check the identity
for $y,z$ in the basis $\{e_i\mid i\in\pm I\}$, and by symmetry we may assume
$y=e_0$ or $e_1$. We first compute the following:
\begin{eqnarray}
xe_0&=&\lambda_0e_0+\lambda_1e_1+\lambda_\omega e_{\omega} + \lambda_{\ombar}
e_{\ombar}\cr
e_0x&=&\lambda_0e_0+\lambda_{-1}e_{-1}+\lambda_{-\omega} e_{-\omega} 
+ \lambda_{-\ombar}
e_{-\ombar}\cr
(xe_0)x=x(e_0x) &=& \lambda_0^2e_0 - (\lambda_1\lambda_{-1}+\lambda_{\omega}\lambda_{-\omega}
+\lambda_{\ombar}\lambda_{-\ombar})e_{-0}\cr
&&+\lambda_0(\lambda_1e_1+\lambda_{\omega}e_{\omega}
+\lambda_{\ombar}e_{\ombar}+\lambda_{-1}e_{-1}+\lambda_{-\omega}e_{-\omega}
+\lambda_{-\ombar}e_{-\ombar})\cr
xe_1&=&-\lambda_{-1}e_0+\lambda_{-0}e_1+
\lambda_{\ombar} e_{-\omega} - \lambda_{\omega}
e_{-\ombar}\cr
e_1x&=&-\lambda_{-1}e_{-0}+\lambda_{0}e_1-
\lambda_{\ombar} e_{-\omega} + \lambda_{\omega}
e_{-\ombar}\cr
(xe_1)x=x(e_1x)&=& \lambda_{-1}^2e_{-1} + (\lambda_0\lambda_{-0}+
\lambda_{\omega}\lambda_{-\omega}+\lambda_{\ombar}\lambda_{-\ombar})e_1\cr
&&+\lambda_{-1}(\lambda_0e_0+\lambda_{-\omega}e_{-\omega}+\lambda_{-\ombar}e_{-\ombar}
-\lambda_{-0}e_{-0}-\lambda_{\omega}e_{\omega}-\lambda_{\ombar}e_{\ombar})
\end{eqnarray}
In particular, we deduce by linearity that $(xy)x=x(yx)$ for all $x,y\in \mathbb O$.
(Similar calculations show that $x(xy)=(xx)y$ and $(yx)x=y(xx)$.)
We now have to calculate the left-hand side of the identity in the following cases,
and check equality with the right-hand side, which is either zero or given above:
\begin{eqnarray}
y=e_0,z=e_0,&&yz=e_0\cr
y=e_0,z=e_{-0}&& yz=0\cr
y=e_0,z=e_1,&&yz=0\cr
y=e_{-0},z=e_{1},&&yz=e_1\cr
y=e_1,z=e_1,&&yz=0\cr
y=e_1,z=e_{-\omega},&&yz=0\cr
y=e_{-1},z=e_{1},&&yz=e_0\cr
y=e_{-\omega},z=e_{-\ombar},&&yz=e_{1}
\end{eqnarray}
These calculations are left to the reader.
\end{proof}
\begin{definition}
Let $\mathbb J=\mathbb J_F$ be the set of $3\times 3$ Hermitian matrices with entries in $\mathbb O_F$, that 
is matrices
\begin{eqnarray}
X=
(a,b,c\mid A,B,C) &=& \begin{pmatrix}a&C&\overline{B}\cr \overline{C}&b&A\cr
B&\overline{A}&c\end{pmatrix}
\end{eqnarray}
with $a=\overline{a}$, $b=\overline{b}$, $c=\overline{c}$.
The \emph{trace} of $X$ is $\Tr(X)=a+b+c$,
the \emph{norm} of $X$ is
\begin{eqnarray}
Q(X)&=&A\overline{A}+B\overline{B}+C\overline{C}-ab-ac-bc
\end{eqnarray}
and the \emph{Dickson--Freudenthal determinant} of $X$ is
\begin{eqnarray}
\det(X) &=& abc-aA\overline{A}-bB\overline{B}-cC\overline{C}
+\Tr(ABC).
\end{eqnarray}
\end{definition}
(The definition of the determinant in \cite{FSG}, in (4.130) and elsewhere,
is wrong.)
Notice that we are \emph{not} defining a Jordan product on $\mathbb J$, so
$\mathbb J$ is \emph{not} a Jordan algebra. 

We show next that the Dickson--Freudenthal determinant as defined here is equivalent to
Dickson's original cubic form \cite{Dickson1}
in $27$ variables. First define $27$ variables $a,b,c,A_i,B_i,C_i$, 
where $A=\sum_{i\in \pm I} A_ie_i$
and similarly for $B_i$ and $C_i$. Then we calculate the determinant as
\begin{eqnarray}
\det(X)&=&abc-\sum_{i\in I
} (aA_iA_{-i}+bB_iB_{-i}+cC_iC_{-i})\cr &&\qquad
+\sum_{i+j+k=\pm 0}(\Tr(e_ie_je_k))A_iB_jC_k,
\end{eqnarray}
where the coefficients $\Tr(e_ie_je_k)$ 
of the $32$ terms in the last sum are all $\pm 1$.
Further calculation gives $$\Tr(e_ie_je_k)=+1$$ when $(i,j,k)$ is a cyclic rotation of a multiple of
$(0,0,0)$ or $(1,\ombar,\omega)$, and  $$\Tr(e_ie_je_k)=-1$$ for cyclic rotations of multiples of
$(1,\omega,\ombar)$ or $(1,0,-1)$.

Dickson's $27$ variables were called $x_i$, $y_j$ and $z_{ij}=-z_{ji}$, where $i,j
\in\{1,2,3,4,5,6\}$, and the cubic form is
\begin{eqnarray}
&&\sum_{i,j} x_iy_jz_{ij} + \sum z_{ij}z_{kl}z_{mn}
\end{eqnarray}
where the second sum is over all 
partitions $\{\{i,j\},\{k,l\},\{m,n\}\}$ of $\{1,2,3,4,5,6\}$, ordered so that
$ijklmn$ is an even permutation of $123456$.

To translate between the two cubic forms, let $a=z_{13}$, $b=z_{26}$, $c=z_{45}$,
and the other $24$ variables as follows.
\begin{eqnarray}
\begin{array}{c|ccc|ccc|}
i&A_i&B_i&C_i&A_{-i}&B_{-i}&C_{-i}\cr\hline
0&z_{25}&z_{43}&z_{16}&z_{46}&z_{15}&z_{23}\cr
1&y_3&y_6&y_5&x_1&x_2&x_4\cr
\omega&x_3&x_6&x_5&-y_1&-y_2&-y_4\cr
\ombar & z_{56}&z_{35}&z_{63}&z_{42}&z_{14}&z_{21}\cr\hline
\end{array}
\end{eqnarray}
Observe that the symmetry $(a,b,c)(A,B,C)$ corresponds to $(1,2,4)(3,6,5)$,
so that we only need to check $17$ of the $45$ terms. 
The (easy) calculations are omitted---in fact the determinant is
exactly the negative of Dickson's cubic form.

Hence we may
interpret the determinant of the split Jordan algebra
as a cubic form over any field, and then follow Dickson and define
$SE_6(q)$ for any $q$ to be the group of $\mathbb F_q$-linear maps which preserve
this cubic form over $\FF_q$.
Similarly, we may define $F_4(q)$ to be the subgroup of $SE_6(q)$
consisting of those maps which fix the identity element.
Notice in particular that we now have a definition of $F_4(q)$ in characteristic $2$
which completely avoids the need for introducing the `quadratic Jordan algebras'
of McCrimmon \cite{McCrimmon}.

\section{Some elements of $E_6(q)$}
In this section we write down some elements of $SE_6(q)$, which we shall later
show are enough to generate the whole group. All these elements will be encoded as
$3\times 3$ matrices $M$, written over some commutative subring of $\mathbb O$, and
acting on $X\in \mathbb J$ via $X\mapsto \overline{M}^\top X M$.
In fact, most of the proofs in this section also work for arbitrary octonion
algebras over arbitrary fields.

First 
observe that the coordinate permutations, generated by
\begin{eqnarray}
(a,b,c\mid A,B,C)&\mapsto&(c,a,b\mid C,A,B)\cr
(a,b,c\mid A,B,C)&\mapsto&(a,c,b\mid\overline{A},\overline{C},\overline{B})
\end{eqnarray}
preserve the determinant. These are encoded respectively by the matrices
$$\begin{pmatrix}0&1&0\cr 0&0&1\cr 1&0&0\end{pmatrix}, 
\begin{pmatrix}1&0&0\cr 0&0&1\cr 0&1&0\end{pmatrix}.$$

\begin{lemma}
\label{tvpresdet}
 Let
$$M_x=\begin{pmatrix}1&x&0\cr 0&1&0\cr 0&0&1\end{pmatrix},$$
for any $x\in\mathbb O$.
If $X =(a,b,c\mid A,B,C)$ then 
$$\overline{M_x}^\top XM_x=(a,ax\overline{x}+b+(\overline{x}C+\overline{C}x),c\mid A+
\overline{x}.\overline{B},B,ax+C),$$
and $\det(\overline{M_x}^\top XM_x) = \det(X)$.
\end{lemma}
\begin{proof}
The calculation of $\overline{M_x}^\top XM_x$ is an easy exercise.
The individual terms of the determinant are as follows:
\begin{eqnarray}
abc&\mapsto& abc + a^2cx\overline{x} + ac(\overline{x}C+\overline{C}x)\cr
-aA\overline{A}&\mapsto& -aA\overline{A}-aA(Bx)-a(\overline{x}.\overline{B}).\overline{A}-ax\overline{x}\overline{B}B\cr
-bB\overline{B}&\mapsto&-bB\overline{B}-ax\overline{x}B\overline{B}-(\overline{x}C+\overline{C}x)B\overline{B}\cr
-cC\overline{C}&\mapsto&-cC\overline{C}-ac(\overline{x}C+\overline{C}x)-
a^2cx\overline{x}\cr
(AB)C&\mapsto& (AB)C+(\overline{B}B)\overline{x}C+ax\overline{x}B\overline{B} +a(AB)x\cr
\overline{C}.(\overline{B}.\overline{A})&\mapsto&
\overline{C}.(\overline{B}.\overline{A})+ax\overline{x}\overline{B}B+(B\overline{B})\overline{C}x
+a\overline{x}(\overline{B}.\overline{A})
\end{eqnarray}
and it is easy to see that all the terms on the right-hand side
cancel out, except those in $\det(a,b,c\mid A,B,C)$.
\end{proof}

Now if two matrices $M$ and $N$ both lie in $SE_6(q)$, and are both written over
the same $2$-dimensional subring of the octonions, then there is sufficient
associativity to show that the action of $M$ followed by the action of $N$ is the same as
the action of $MN$, that is
$$(\overline{MN})^\top X(MN) = \overline{N}^\top(\overline{M}^\top X M) N.$$ 
In other words, we can multiply together the generators of
$SE_6(q)$ as long as the entries stay within the same $2$-dimensional subring.

In this way we obtain $48$ root groups by putting $x=\lambda e_i$
(for arbitrary $\lambda\in F$ and fixed $i$) in one of the six off-diagonal positions.

Indeed, more is true. If we apply the matrices $M_x$ and $M_y$ in turn
to $X$ we obtain
\begin{eqnarray}
&&(a, b + ax\overline{x} + ay\overline{y} + (\overline{x}C+\overline{C}x)
+ (\overline{y}(ax+C)+(\overline{ax+C})y),c \mid\cr&&
\qquad A+\overline{x}\overline{B}+\overline{y}\overline{B},C+ax+ay),
\end{eqnarray}
which is the same as the image of $X$ under the action of $M_{x+y}$. Thus the matrices $M_x$ generate
an elementary abelian group of order $q^8$. Similarly, if we follow $M_x$ by
$$\begin{pmatrix}1&0&y\cr 0&1&0\cr 0&0&1\end{pmatrix},$$
we obtain
$$(a,b+ax\overline{x}+\overline{x}C+\overline{C}x,
c+ay\overline{y}+By+\overline{y}\overline{B}\mid
A+\overline{x}\overline{B}+\overline{C}y, B+a\overline{y}, C+ax)$$
so in fact we obtain an elementary abelian group of order $q^{16}$ in this way.

More elements may be obtained by the following computations in one of the
$2\times 2$ blocks. If $u\in\mathbb O$ is invertible, then we have
$$\begin{pmatrix}1&u-1\cr 0&1\end{pmatrix}\begin{pmatrix}1&0\cr 1&1\end{pmatrix}
\begin{pmatrix}1&{u}^{-1}-1\cr 0&1\end{pmatrix}\begin{pmatrix}1&0\cr-u&1\end{pmatrix} =
\begin{pmatrix}u&0\cr 0&{u}^{-1}\end{pmatrix}.$$
Hence
the group contains the diagonal matrices
$$M=\diag(u,\overline{u},1)=\begin{pmatrix}u&0&0\cr 0&\overline{u}&0\cr 0&0&1\end{pmatrix},$$
where $u\in\mathbb O$ satisfies $u\overline{u}=1$, which acts on $\mathbb J$ as 
$$(a,b,c\mid A,B,C)\mapsto(a,b,c\mid uA, Bu, \overline{u}C\overline{u}).$$
By using the Moufang law one can show directly that these matrices
preserve the determinant, though of course this follows from
the calculations already done. 
Since we have $(uA)(Bu)=u(AB)u$, 
repeated use of the identities $\Tr(AB)=\Tr(BA)$ and $\Tr(A(BC))=\Tr((AB)C)$
implies that $$\Tr((uA)(Bu)(\overline{u}C\overline{u}))=\Tr(ABC).$$
The other terms in the determinant are easy to deal with.

Next we analyse the group generated by these diagonal matrices.
Consider the action on $C$, that is the map $C\mapsto \overline{u}C\overline{u}$.
Since reflection in $1$ is the map $x\mapsto -\overline{x}$, reflection  in $u$ is the
map $y\mapsto -\overline{u}y\overline{u}$, and the given action is the composition of
these two maps. As $u$ ranges over all octonions of norm $1$, therefore, the action
generated is that of $\Omega_8^+(q)$.
Indeed, by using all of the diagonal matrices we can get a similar result for reflections in
vectors $u$ of arbitrary norm, and hence get an
action of $\SO_8^+(q)$. The kernel of this action is given by $u\in \mathbb F_q$,
and thus we have an action of a group of shape $C_{q-1}.\SO_8^+(q)$ 
on $\mathbb J$.

Now extend this to the action on the $10$-space of matrices  of the form
$$(a,b,0\mid 0,0,C).$$ The elements $M_x$ and their transposes extend the
action $\SO_8^+(q)$ to $\SO_{10}^+(q)$, preserving the norm
$C\overline{C}-ab$. Again we have a kernel of order $q-1$, giving a group
of shape $C_{q-1}.\SO_{10}^+(q)$.

\section{The white points}
In order to calculate the group order we count the `rank 1' matrices, otherwise
known as the `white' vectors.
In order to obtain a construction which works also in characteristics 
$2$ and $3$, 
we define these purely in terms of the determinant.
\begin{definition}
For a fixed non-zero $W\in\mathbb J$, the expression 
$\det(W+X)$ is a cubic form in the variables of $X$,
and has a cubic term $\det(X)$, a quadratic term, a linear term, and a constant term
$\det(W)$.
\begin{enumerate}
\item If the linear term is identically zero, then $W$ is called \emph{white}.
\item If the constant term $\det(W)$ is non-zero, then $W$ is called \emph{black}.
\item Otherwise, $W$ is called \emph{grey}.
\end{enumerate}
A \emph{white/grey/black point} is a $1$-dimensional subspace spanned by a 
white/grey/black vector.
\end{definition}
By analogy with ordinary $3\times 3$ matrices, we may think of white, grey and black
matrices as having \emph{rank} $1,2,3$ respectively.
For example, $(1,0,0\mid0,0,0)$ is white because
\begin{eqnarray}
\det(1+a,b,c\mid A,B,C) &=& bc-A\overline{A}+\det(a,b,c\mid A,B,C)
\end{eqnarray}
has zero linear term. Similarly, $(1,1,1\mid 0,0,0)$ is black because it has
determinant $1$. Finally, $(1,1,0\mid 0,0,0)$ is grey because
\begin{eqnarray}
\det(1+a,1+b,c\mid A,B,C) &=& c + (a+b)c -A\overline{A}-B\overline{B}
\cr&&\qquad+\det(a,b,c\mid A,B,C)
\end{eqnarray}
has zero constant term but non-zero linear term.
The terms white, grey and black were introduced by Cohen and Cooperstein \cite{CC}.
Jacobson \cite{Jac3} uses the equivalent terms rank 1, rank 2 and rank 3.
Aschbacher \cite{Asch1} calls them respectively singular, brilliant non-singular, and dark.

\begin{lemma}\label{whitedesc}
A non-zero element $(a,b,c\mid A,B,C)$
of $\mathbb J$ 
is {white} if and only if one of the
following holds:
\begin{enumerate}
\item at least one of the diagonal entries (say $c$) is
non-zero, and $(a,b,c\mid A,B,C)$ is of the form $c\overline{v}^\top v$, where
$v=(x,y,1)=(B/c,\overline{A}/c,1)$, or
\item $a=b=c=0$,
$A\overline{A}=B\overline{B}=C\overline{C}=0$
and $AB=BC=CA=0$. 
\end{enumerate}
\end{lemma}
\begin{proof}
Suppose that $W=(a,b,c\mid A,B,C)$ is white, and let $X=(p,q,r\mid P,Q,R)$, so that
the terms in $\det(W+X)$ which are linear in $p,q,r,P,Q,R$ are
\begin{eqnarray}
&&bcp+acq+abr-A\overline{A}p-B\overline{B}q-C\overline{C}r\cr
&&{}-a(P\overline{A}+A\overline{P}) - b(Q\overline{B}+B\overline{Q})
-c(R\overline{C}+C\overline{R})\cr
&&{}+\Tr(PBC +QCA + RAB)
\end{eqnarray}
This can be re-written as
\begin{eqnarray}
&&(bc-A\overline{A})p+(ac-B\overline{B})q+(ab-C\overline{C})r\cr
&&{}+\Tr((BC-a\overline{A})P +(CA- b\overline{B})Q
+(AB-c\overline{C})R)
\end{eqnarray}
For this to be identically zero, it is necessary and sufficient that 
the
following equations be satisfied:
\begin{eqnarray}
bc&=& A\overline{A},\cr
ac&=&B\overline{B},\cr
ab&=&C\overline{C},\cr
BC&=&a\overline{A},\cr
CA&=&b\overline{B},\cr
AB&=&c\overline{C}.
\end{eqnarray}
Now if any of $a,b,c$ is non-zero, say $c\ne 0$, we have
\begin{eqnarray}
b&=&A\overline{A}/c\cr a&=&B\overline{B}/c\cr \overline{C}&=&AB/c,
\end{eqnarray}
 and hence
\begin{eqnarray}
\begin{pmatrix}a & C & \overline{B}\cr \overline{C} & b & A\cr B & \overline{A} & c\end{pmatrix}
&=& \frac1c\begin{pmatrix}\overline{B}\cr A \cr c\end{pmatrix}.\begin{pmatrix}B & \overline{A} & c\end{pmatrix}.
\end{eqnarray}
On the other hand, if $a=b=c=0$, then the equations reduce to
$$A\overline{A}=B\overline{B}=C\overline{C}=AB=BC=CA=0.$$
\end{proof}
\begin{theorem}\label{numwhite}
The number of white  vectors is $(q^9-1)(q^8+q^4+1)$.
\end{theorem}
\begin{proof}
First suppose Lemma~\ref{whitedesc}(i) holds, and 
disjoin cases according to how many of $a,b,c$ are non-zero.  If
all three of $a,b,c$ are non-zero, then there are $q-1$ choices
for each of $a,b,c$, and $q^7-q^3$ choices for each of $x,y$, making
$$(q-1)^3(q^7-q^3)^2$$ such vectors in all. 
If just two of them are non-zero, say $b$ and $c$,
then there are $q^7-q^3$ choices for $y$ and $q^7+q^4-q^3$ choices for $x$
(any isotropic octonion, or $0$), making $$3(q-1)^2(q^7-q^3)(q^7+q^4-q^3)$$ in all.
If just one of them is
non-zero, then there are $q^7+q^4-q^3$ choices 
for each of $x,y$, making $$3(q-1)(q^7+q^4-q^3)^2$$ in all.

In the second case of Lemma~\ref{whitedesc}
we disjoin cases according to how many of $A,B,C$ are non-zero.
If all three of $A,B,C$ are non-zero, then there are $$(q^4-1)(q^3+1)=q^7+q^4-q^3-1$$ 
choices for $A$, and the condition $AB=0$
leaves $q^4-1$ choices for $B$. The conditions $BC=0$ and $CA=0$ leave $q^3-1$
choices for $C$, making $$(q^4-1)^2(q^6-1)$$ in total.
Similarly, if just two of $A,B,C$ are non-zero, there are $$3(q^4-1)^2(q^3+1)$$ choices;
and if just one is non-zero, there are $3(q^4-1)(q^3+1)$ choices.
Adding together these six expressions gives the total
$(q^9-1)(q^8+q^4+1)$ as claimed. 
\end{proof}

For clarity, let us define $\mathcal G=SE_6(q)$, that is the group of $\mathbb F_q$-linear
maps which preserve the determinant, and define $G$ to be the group generated
by the matrices $M_x$, their transposes and images under permutations of the
three coordinates. We have shown that $G\le \mathcal G$. It is our aim to
show that $G=\mathcal G$, and deduce the order of the group from this.

It is 
a straightforward exercise to show that $G$ 
acts transitively
on the set of white 
{points}.
On the other hand, it is obvious from the definition
that $\mathcal G$ 
also preserves
this set.
Hence it is sufficient to show that the stabilizer of a white point in $\mathcal G$
is equal to the stabilizer of a white point in $G$.

\begin{theorem}
The stabilizer in $G$
of a white point is at least a group of shape $q^{16}.C_{q-1}.\mathrm{SO}_{10}^+(q)$, where
in characteristic $2$, we interpret $\SO_{10}^+(q)$ as meaning $\Omega_{10}^+(q)$.
\end{theorem}
\begin{proof}
We consider the stabilizer of the white point spanned by $(1,0,0\mid 0,0,0).$
This is invariant under an elementary abelian 
group of order $q^{16}$ generated by elements of
the shape $$\begin{pmatrix}1&0&0\cr x&1&0\cr 0&0&1\end{pmatrix}\mbox{ and }
\begin{pmatrix}1&0&0\cr 0&1&0\cr y&0&1\end{pmatrix}.$$
Similarly it is invariant
under the action of all diagonal matrices and
$$\begin{pmatrix}1&0&0\cr 0&1&x\cr 0&0&1\end{pmatrix}\mbox{ and }
\begin{pmatrix}1&0&0\cr 0&1&0\cr 0&y&1\end{pmatrix},$$
which together generate $C_{q-1}.\SO_{10}^+(q)$.
Hence
we easily see a subgroup of $G$ 
of shape $q^{16}{:}C_{q-1}.\mathrm{SO}^+_{10}(q)$
fixing this white point.
\end{proof}

Next we show that the suborbits are
the same in both groups.
\begin{lemma}  Given any white point $W$, 
\begin{enumerate}
\item there are exactly $q(q^3+1)(q^8-1)/(q-1)$ white points
$X$ such that all points in $\langle W,X\rangle$ are white.
\item there are exactly $q^8(q^4+1)(q^5-1)/(q-1)$ white points $Y$ such that
$\langle W, Y\rangle$ contains only two white points.
\end{enumerate}
 Moreover, the stabiliser in
$G$ of $W$ acts transitively on the points $X$, and transitively on the points $Y$.
Hence  the permutation actions of $G$ and of $\mathcal G$
 on the white points each have rank $3$, with the given suborbit lengths.
\end{lemma}
\begin{proof}
We may assume that $W$ is spanned by $(1,0,0\mid 0,0,0)$. 
\begin{enumerate}
\item Hence in the first
part we are counting the remaining points spanned by a vector of shape
$(a,0,0\mid 0,B,C)$. As in the
proof of Theorem~\ref{numwhite}, the conditions on $B$ and $C$ result in
the number of solutions for $B$ and $C$ being
$$(q^4-1)^2(q^3+1)+2(q^4-1)(q^3+1)=(q^8-1)(q^3+1).$$
Dividing by $q-1$ for the scalars, and multiplying by $q$ for the choice of $a$,
gives us the result.
\item
Obviously all white points not already counted have the second property.
The number of them is easily computed. 
\end{enumerate}
Transitivity is immediate using the
action of the group $q^{16}{:}C_{q-1}.\mathrm{SO}^+_{10}(q)$ already exhibited.
\end{proof}


\begin{theorem}
The stabilizer in $\mathcal G$
of a white point is at most a group of shape $q^{16}.C_{q-1}.\mathrm{SO}_{10}^+(q)$, where
again we interpret $\SO_{10}^+(q)$ as meaning $\Omega_{10}^+(q)$
in characteristic $2$.
\end{theorem}
\begin{proof}
We again consider the stabilizer of the white point spanned by $$(1,0,0\mid 0,0,0).$$
First note that this stabilizer fixes the $17$-space
of matrices of the form $$(a,0,0\mid 0,B,C).$$ Hence it acts on the $10$-dimensional
quotient space. Now the trilinear form obtained by polarizing the determinant
induces a bilinear form on this quotient, by substituting the original white vector
as the first variable. This bilinear form is invariant up to scalar multiplication, and therefore
the action of the point stabilizer on the $10$-dimensional quotient can be no
bigger than already given. In particular, any element of the kernel of
this action maps $(0,1,0\mid 0,0,0)$ to a matrix of the form $(0,1,0\mid0,0,C)$,
and maps $(0,0,1\mid0,0,0)$ to $(0,0,1\mid0,B,0)$.

But we already have a group of order $q^{16}$ permuting these pairs of matrices regularly,
so we may assume that the two white points spanned by $(0,1,0\mid0,0,0)$ 
and $(0,0,1\mid0,0,0)$ are fixed. Now the white points which are adjacent to both
of these span the $8$-space $\{(0,0,0\mid A,0,0)\}$, so this $8$-space is fixed.
Similarly the $8$-spaces $\{(0,0,0\mid 0,B,0)\}$ and $\{(0,0,0\mid0,0,C)\}$. As the white
points are just the isotropic vectors in these $8$-spaces, the action on any one of them
can be no more than the orthogonal group already exhibited.

Hence we may assume that our element of the kernel acts trivially on one:
say on the $(0,0,0\mid A,0,0)$.
Now we have a large number of pairs of non-adjacent white points which are fixed,
and for every one of these pairs, the $8$-space of white points which are adjacent to
both is also fixed. This is enough to show that the kernel of the action
is no bigger than the group already exhibited.
\end{proof}

As a consequence, we now have: 
\begin{corollary}
$$|SE_6(q)| = q^{36}(q^{12}-1)(q^9-1)(q^8-1)(q^6-1)(q^5-1)(q^2-1).$$
\end{corollary}

Define $E_6(q)$ to be the quotient of $SE_6(q)$ by any scalars it contains. Note
that a scalar $\lambda$ is in $SE_6(q)$ if and only if $\det(\lambda X)=\det(X)$
for all $X$, that is if and only if $\lambda^3=1$. Hence $SE_6(q)$ is a triple cover
of $E_6(q)$ if $q\equiv 1\bmod 3$, and $SE_6(q)\cong E_6(q)$ otherwise.
To prove that $E_6(q)$ is simple, we use Iwasawa's Lemma:
\begin{lemma}
If $G$ is a perfect, primitive permutation group, and the point stabiliser has a normal
abelian subgroup whose $G$-conjugates generate $G$, then $G$ is simple.
\end{lemma}
Now consider the action of $E_6(q)$ on the white points. This action is
obviously primitive and faithful. Now $SE_6(q)$ is generated by the conjugates of
$M_x$, which lies in an abelian normal subgroup of the stabilizer 
$q^{16}.C_{q-1}.\SO_{10}^+(q)$ of a point.
In particular, $M_x$ lies in the derived group, so the group is perfect.
Hence, by Iwasawa's Lemma, $E_6(q)$ is simple.

\section{Building the building} 
The classification of white vectors above allows us to classify the subspaces
which consist entirely of white vectors. 
\begin{theorem}
If $W$ is a subspace of $\mathbb J$ consisting entirely of white vectors (and $0$),
then $W$ is taken by an element of $SE_6(q)$ to one of the following:
\begin{eqnarray}
W_1&=&\langle (1,0,0\mid 0,0,0)\rangle\cr
W_2&=&\langle W_1, (0,0,0\mid 0,e_{-1},0)\rangle\cr
W_3&=&\langle W_2, (0,0,0\mid 0,e_{\ombar},0)\rangle\cr
W_4&=&\langle W_3, (0,0,0\mid 0,e_{\omega},0)\rangle\cr
W_5&=&\langle W_4,  (0,0,0\mid 0,e_0,0)\rangle\cr
W_5'&=&\langle W_4, (0,0,0\mid 0,e_{-0},0)\rangle\cr
W_6&=&\langle W_5', 
(0,0,0\mid 0,0,e_{-1})\rangle
\end{eqnarray}
\end{theorem}

\begin{proof}
First observe that there is a $6$-space consisting entirely of white vectors,
spanned by $$(1,0,0\mid0,0,0), (0,0,0\mid0,0,e_{-1}),(0,0,0\mid 0,B,0),$$ 
where $B\in\langle e_{-1},e_{\ombar},e_{\omega},e_{-0}\rangle,$ which is obviously maximal.
Moreover, the root elements already given act as transvections on this $6$-space $W_6$,
and generate a group which acts on it as $\SL_6(q)$.
Therefore it suffices to prove that every pure white subspace is contained in an image
under the group of $W_5$ or $W_6$.

We have already shown that there is
a unique orbit of the group on pure white $1$-spaces and $2$-spaces, 
so we may take the latter
to be spanned by $(1,0,0\mid0,0,0)$ and $(0,0,0\mid 0,e_{-1},0)$. 
Now all vectors which together with $(1,0,0\mid0,0,0)$ span a pure white $2$-space
are of the form $(a,0,0\mid 0,B,C)$. Therefore
our space
contains white vectors of shape $(0,0,0\mid 0,B,C)$,
where $B$ lies in some totally isotropic subspace of the octonions, which,
using the action of the orthogonal group, may be
taken to be one of 
$$\langle e_{-1},e_{\ombar}\rangle, \quad \langle e_{-1},e_{\ombar},
e_{\omega}\rangle,\quad
\langle e_{-1},e_{\ombar},e_{\omega},e_{0}\rangle,\quad 
\langle e_{-1},e_{\ombar},e_{\omega},e_{-0}\rangle.$$ Then $C$ lies
in the corresponding annihilator $\langle e_{-1},e_{\ombar}\rangle$ (in the first case) or
$\langle e_{-1}\rangle$ (in the second and last cases) or $0$ (in the third case).
All of these contain at least a $4$-space in common with $W_6$,
and by transitivity on these $4$-spaces, we see that there
is just one more
orbit on  maximal white subspaces, with representative the $5$-space
spanned by $(1,0,0\mid0,0,0)$ and $(0,0,0\mid0,B,0)$ with 
$B\in\langle e_{-1},e_{\ombar},e_{\omega},e_0\rangle$.
Since any pure white $4$-space is contained in a unique pure white $6$-space,
the result follows.\end{proof}

By adjoining appropriate root groups to the subgroup of
$2^2.\POmega_8^+(q).S_3$ which fixes $W_i$ or $W_5'$, it is easy to obtain
generators for the stabilizers. For $i=1,2,3,5,6$, these turn out to be five of the six maximal
parabolic subgroups. The other maximal parabolic subgroup fixes the
$10$-space $W_{10}$ defined by adjoining to $W_5$ the $5$-space
spanned by $(0,1,0\mid0,B,0)$ with $B\in\langle e_{-0},e_{-\omega},e_{-\ombar},e_1\rangle$.
In other words, $$W_{10}=\{(a,0,0\mid 0,B,C)\}.$$
Notice that $W_{10}$ has a quadratic form defined on it, which is invariant
up to scalar multiplication. With respect to this form, the white points are
isotropic, while the non-isotropic points are grey.


\section{Duality and the subgroup $F_4(q)$}
There is a second, `dual', action of $SE_6(q)$ on the set $\mathbb J$
of $3\times 3$
octonion Hermitian matrices, whereby a matrix $M$ acts as
$$M:X\mapsto M^{-1}X(\overline{M}^\top)^{-1}.$$
To see that this is an action, we need to show that any relation between the
original actions of $M$ by
$$M:X\mapsto \overline{M}^\top X M$$
also holds for $(\overline{M}^\top)^{-1}$. But by symmetry, any word in
the original (right-)actions of the $M$ corresponds to the reverse word in
 the (left-)actions of the corresponding $\overline{M}^\top$. In particular,
given any relator satisfied by the actions of matrices $M_i$, the reverse relator
is satisfied by the corresponding $\overline{M_i}^\top$.
Hence the original relator is satisfied by the $(\overline{M_i}^\top)^{-1}$.

This implies that the map $M\mapsto (\overline{M}^\top)^{-1}$ on the
given generators of $SE_6(q)$ induces an automorphism of $SE_6(q)$.
It is easy to see that it is not inner, so we shall call it \emph{duality}.

Now if $M$ is a generator of $SE_6(q)$ fixed by this duality automorphism,
then $M=(\overline{M}^\top)^{-1}$, so $\overline{M}^\top IM=I$.
In other words, $M$ fixes the identity element of $\mathbb J$, so
$M$ lies in $F_4(q)$.

For example, the diagonal elements $\diag(u,\overline{u},1)$
with $u\overline{u}=1$ satisfy this condition.
So do the elements
$$\begin{pmatrix}1&x&0\cr-\overline{x}&1&0\cr0&0&1\end{pmatrix},$$
provided $x\overline{x}=0$. These elements are in $SE_6(q)$ because
$$\begin{pmatrix}1&x\cr-\overline{x}&1\end{pmatrix}=\begin{pmatrix}1&x\cr 0&1\end{pmatrix}
\begin{pmatrix}1&0\cr -\overline{x} & 1\end{pmatrix}.$$
Hence by putting $x=\lambda e_i$ we obtain $8$ root
groups, becoming $24$ when we allow coordinate permutations as well.
(The other $24$ root groups lie in the subgroup generated by the diagonal matrices.)
The case $x=\lambda e_0$ realises precisely the short root element displayed in
(4.105) of \cite{FSG}. 

 As noted in \cite{FSG},
the normalizer of a maximal torus can be found inside a subgroup of shape
$2^2.\POmega_8^+(q).S_3$.
If we take the diagonal elements $\diag(u,\overline{u},1)$ 
and $\diag(1,u,\overline{u})$ with $u=\lambda e_{-i}+\lambda^{-1}e_i,$
and adjoin the coordinate permutations,
then we obtain the normalizer of a maximal split torus.

The long root elements also lie in  $2^2.\POmega_8^+(q)$. For example we may take
the product of the three group elements given by
\begin{eqnarray}
&&\diag(1+e_{-1},1-e_{-1},1),\cr
&& \diag(1-\lambda e_{\ombar},1+\lambda e_{\ombar},1), \cr
&&\diag(1-e_{-1}+\lambda e_{\ombar},1+e_{-1}-\lambda e_{\ombar},1)
\end{eqnarray} 
to give the long root element
displayed in (4.104) of \cite{FSG}.


One way to compute the order of $F_4(q)$ is to count the primitive idempotents.
The official definition in terms of the Jordan algebra is that they are
idempotents $X$ (in the sense that $X\circ X=X$) with trace $1$.
However, this definition does not necessarily
work in characteristic $2$ or $3$, and they may alternatively 
be defined as white vectors with trace $1$, so that it is not necessary to treat these
characteristics differently.
\begin{definition}
An element of $\mathbb J$ is called a \emph{primitive idempotent}
if it is a white vector with trace $1$.
\end{definition}
A straightforward calculation shows that there are precisely
$$q^8(q^8+q^4+1)$$ primitive idempotents, 
and $$(q^8-1)(q^8+q^4+1)=(q^{12}-1)(q^4+1)$$ white vectors of trace $0$. More precisely, 
the trace can be non-zero only in the
first three of the six cases in the proof of Theorem~\ref{numwhite}.
The number of choices of $a,b,c$ which give trace $1$ is $q^2-3q+3$
if all are non-zero, and $3(q-2)$ if two are non-zero, and $3$ if just one is non-zero.
Hence the total number of primitive idempotents is
$$(q^2-3q+3)(q^7-q^3)^2
+3(q-2)(q^7-q^3)(q^7+q^4-q^3)
+3(q^7+q^4-q^3)^2$$
which simplifies to $q^8(q^8+q^4+1)$.
Subtracting $q-1$ times this
from the total number of white vectors gives the number with trace $0$.

It is now clear that $F_4(q)$ acts transitively on the primitive idempotents.
To calculate the group order we need only calculate the order of the stabilizer
of one of the primitive idempotents. 
We already know that
the stabilizer in $SE_6(q)$ of a white point is a group of shape
$q^{16}.C_{q-1}.\SO_{10}^+(q)$, so we just need to calculate the subgroup
of this which preserves the identity element of $\mathbb J$. 
In the case of a trace $1$ white point, such as $(1,0,0\mid 0,0,0)$,
this is easily seen
to be a subgroup $\Spin_9(q)$.
%
In particular the
formula for the group order is, independently of the characteristic,
$$|F_4(q)|=q^{24}(q^{12}-1)(q^8-1)(q^6-1)(q^2-1).$$

Note also that the stabilizer of a white point of trace $0$ is a group of shape
$$q^{7+8}.C_{q-1}.\SO_7(q),$$ which is one of the maximal parabolic
subgroups of $F_4(q)$. To prove simplicity of $F_4(q)$, we can apply Iwasawa's Lemma
to the action on the white points of trace $0$. Details are given
in Section 4.8.7 of \cite{FSG}.

\section{The compact real form of $E_6$}
We constructed the finite groups $E_6(q)$ by analogy with the
split real form of $E_6$, which is defined in terms of the exceptional
Jordan algebra over the split octonions. In a similar way, we shall
construct the finite groups ${}^2E_6(q)$ by analogy with the
compact real form of $E_6$. But we have seen that
in order to describe the latter, it is not sufficient just to replace the split
octonions by the compact octonions. Instead,
we must first extend the scalars from $\mathbb R$ to $\mathbb C$ to obtain
the complexification $SE_6(\mathbb C)$. 
Then we compactify by decreeing that a certain Hermitian form be
invariant. 
Now, just as in the construction of the unitary groups, there is some
choice as to which Hermitian form to use.
It is not obvious a priori which (if any) is the `best'.

First notice that, since there is only one isomorphism type of
complex octonion algebra, we can take either the basis $\{e_i\mid i\in\pm I\}$
defined above for the `split' real octonions, 
or the  basis $\{1=i_\infty,i_0,i_1,\ldots,i_6\}$ for the `compact' real octonions,
or Cayley numbers. We can switch
between the two by a base change such as the following:
\begin{eqnarray}
1&=&e_0+e_{-0},\cr
i_1&=&e_{\omega}+e_{-\omega},\cr
i_2&=&e_{\ombar}+e_{-\ombar},\cr
i_4&=&e_{-1}+e_1,\cr
ji_3&=& e_0-e_{-0},\cr
ji_0&=&e_{\omega}-e_{-\omega},\cr
ji_5&=&e_{\ombar}-e_{-\ombar},\cr
ji_6&=&e_{-1}-e_1,
\end{eqnarray}
where $j$ denotes $\sqrt{-1}$ in the scalar copy of $\mathbb C$.

Now there is an obvious Hermitian form $h$ on the complex octonions 
obtained by defining
the basis $\{e_i\}$ to be orthonormal.
There is another obvious Hermitian form $h_2$ defined by saying that the basis $\{i_t\}$ is
orthonormal. We show next that $h=2h_2$.

Let us write $x'$ for the complex conjugate $a-bj$ of $x=a+bj$.
Now there are two different ways we might want to extend $'$ to the octonions.
Given an octonion
$$A=\sum_t \alpha_t i_t = \sum_i \beta_i e_i$$
we define
\begin{eqnarray}
A^*&=&\sum_t \alpha_t' i_t\cr
A'&=&\sum_i \beta_i' e_i
\end{eqnarray}
Since $e_i'=e_{-i}$ we have
$$A^*=\sum_i \beta_i' e_{-i}$$
and since $i_t^*=i_t$ for $t=\infty,1,2,4$ and $i_t^*=-i_t$ otherwise, we have
$$A'= \sum_{t=\infty,1,2,4}\alpha_t' i_t - \sum_{t=0,3,5,6}\alpha_t' i_t.$$  
Then we can compute
$$h_2(A)=\sum_t \alpha_t\alpha_t' = (A\overline{A}^*+A^*\overline{A})/2
=(\overline{A}A^*+\overline{A}^*A)/2$$ 
and therefore
$$h(A)=\sum_i \beta_i\beta_i' = A\overline{A}^*+A^*\overline{A}
=\overline{A}A^*+\overline{A}^*A.$$

So now define an Hermitian form $H$ on
$\mathbb J_{\mathbb C}= \mathbb C^{27}$
by $$H(a,b,c\mid A,B,C) = aa'+bb'+cc'+h(A)+h(B)+h(C).$$
Thus $(1,0,0\mid 0,0,0)$, $(0,0,0\mid e_i,0,0)$ and rotations of these form an
orthonormal basis. Then the subgroup of $SE_6(\mathbb C)$ which
preserves this Hermitian form is in fact the compact real form of $SE_6$.
If we prefer to use the basis $\{i_t\}$ for the octonions,
then we have
$$H(a,b,c\mid A,B,C) = aa'+bb'+cc'+h_2(A)+h_2(B)+h_2(C)+
h_2(\overline{A})+h_2(\overline{B})+h_2(\overline{C}).$$
This shows that $H$ is the `natural' Hermitian form induced
on $\mathbb J_{\mathbb C}$ by the `natural' Hermitian form
$h_2$
on the complex octonions.

We now extend $*$ also to $\mathbb J_{\mathbb C}$ by
defining
$$X^*=(a',b',c'\mid A^*,B^*,C^*)$$
for $X=(a,b,c\mid A,B,C)$, and then we find that
\begin{eqnarray}
\Tr(X\circ X^*) &=& aa'+bb'+cc'+h(A)+h(B)+h(C)\cr
&=& H(X).
\end{eqnarray}
 
Now to generate the group we first collect the original generators for
$F_4$, which were diagonal matrices $\diag(u,\overline{u},1)$ with
$u\overline{u}=1$, together with $2\times 2$ matrices acting on two of the
three coordinates as
$\begin{pmatrix}\alpha & \beta \cr -\beta & \alpha\end{pmatrix}$ or
$\begin{pmatrix}\alpha & \beta i_t\cr \beta i_t & \alpha\end{pmatrix}$, for $\alpha^2+\beta^2=1$
and $t\ne\infty.$
Now we can adjoin two further real dimensions of diagonal matrices
by taking $\diag(\alpha,\beta,\gamma)$ with $\alpha,\beta,\gamma\in\mathbb C$
satisfying $\alpha\alpha'=\beta\beta'=\gamma\gamma'=\alpha\beta\gamma=1$.
Then we take the $2\times 2$ matrices
$\begin{pmatrix}\alpha & \beta \cr -\beta' & \alpha'\end{pmatrix}$ and
$\begin{pmatrix}\alpha & \beta i_t\cr \beta' i_t & \alpha'\end{pmatrix}$, for $\alpha\alpha'+\beta\beta'=1$
and $t\ne\infty.$
This gives us the full dimension $78$ for $E_6$ made up of $52$ for $F_4$,
an extra $2$ for diagonal matrices, and $24$ for the $2\times 2$ matrices.

To see that these matrices belong to $E_6(\mathbb C)$, observe that all the matrix
entries lie in $\mathbb C(i_t)$, and that the determinant is $1$. Hence they are
products of the `transvection' generators for $E_6(\mathbb C)$. To see that they also
preserve the Hermitian form $H$, we do some explicit calculations.
We first prove a small lemma. 
\begin{lemma}
If $x,y,z \in \mathbb O$, then
\begin{enumerate}
\item $x(yx)=\Tr(yx).x - (x\overline{x}).\overline{y}$;
\item $\Tr((xy)(z\overline{x}))=x\overline{x}\Tr(yz)$.
\end{enumerate}
\end{lemma}
\begin{proof}
\begin{enumerate}
\item $x(yx-\Tr(yx))=-x(\overline{yx})=-x(\overline{x}.\overline{y})=-(x\overline{x})
\overline{y}$.
\item $\Tr((xy)(z\overline{x}))=\Tr((z\overline{x})(xy))=\Tr(((z\overline{x})x)y)
=\Tr((z(\overline{x}x))y)=x\overline{x}\Tr(yz)$.
\end{enumerate}
\end{proof}

Now we are mapping by a matrix of the form
$$\begin{pmatrix}x&y&0\cr-\overline{y}^* & x' & 0\cr 0&0&1\end{pmatrix}$$
where $x$ is real and $y$ is a multiple of $i_t$ for some $t$.
By explicit computation we see that the action of 
this matrix on $X=(a,b,c\mid A,B,C)$ is given by
\begin{eqnarray}
a&\mapsto&ax^2-x\Tr(C\overline{y}^*)+by^*\overline{y}^*\cr 
b&\mapsto&b(x')^2+x'\Tr(\overline{C}y)+ay\overline{y},\cr
c&\mapsto &c,\cr
A&\mapsto&x^*A+\overline{y}\overline{B},\cr
B&\mapsto& Bx-\overline{A}\overline{y}^*,\cr
C&\mapsto&xx^*C+axy-bx^*y^*-y^*\overline{C}y.
\end{eqnarray}
Then we can compute the new value of $H$ term by term as follows.
First consider the terms in $A$ and $B$.
\begin{eqnarray}
\Tr(A\overline{A}^*)&\mapsto &
xx'\Tr(A\overline{A}^*) + x\Tr(A^*By) + x'\Tr(AB^*y^*)\cr&&\qquad
+\Tr((\overline{y}\overline{B})(B^*y^*))\cr
\Tr(B\overline{B}^*)&\mapsto& xx'\Tr(B\overline{B}^*)
-x\Tr(ByA^*) - x^*\Tr(B^*y^*A)\cr&&\qquad
+\Tr((\overline{A}\overline{y}^*)(yA^*))
\end{eqnarray}
and by applying the lemma and using the fact that $xx^*+y\overline{y}^*=1$
we see that the sum of these two terms is preserved. Now $cc'$ is fixed, and the
other terms are as follows. (Note that all calculations are in the (complex)
quaternion subalgebra generated by $C$ and $y$, so we can use associativity.)
\begin{eqnarray}
aa'&\mapsto&x^2x'^2aa' + x^2y\overline{y}ab'-x^2x'\Tr(\overline{C}^*y)a\cr
&&+x'^2y^*\overline{y}^*a'b + y\overline{y}y^*\overline{y}^*bb'
-y^*\overline{y}^*x'\Tr(\overline{C}^*y)b\cr
&&-xx'^2\Tr(C\overline{y}^*)a' -y\overline{y}x\Tr(C\overline{y}^*)b' +xx'\Tr(C\overline{y}^*)\Tr(\overline{C}^*y)\cr
bb'&\mapsto&y\overline{y}y^*\overline{y}^*aa' + y\overline{y}x^2ab' 
+ y\overline{y}x\Tr(C^*\overline{y}^*)a\cr
&&+x'^2y^*\overline{y}^*a'b + (xx')^2bb' + xx'^2\Tr(C^*\overline{y}^*)b\cr
&&+x'y^*\overline{y}^*\Tr(\overline{C}y)a' + x^2x'\Tr(\overline{C}y)b'
+ xx'\Tr(\overline{C}y)\Tr(C^*\overline{y}^*)\cr
\Tr(C\overline{C}^*)&\mapsto& \Tr(
xx'y\overline{y}^*aa' -x^2y\overline{y}ab' + x^2x'ya\overline{C}^* - xy\overline{y}^*C^*
\overline{y}a\cr
&&-x'^2y^*\overline{y}^*a'b + xx'y^*\overline{y}
bb' -xx'^2y^*\overline{C}^*b + x'y^*\overline{y}^*C^*\overline{y}b\cr
&&+xx'^2C\overline{y}^*a' - x^2x'C\overline{y}b' + (xx')^2C\overline{C}^*
-xx'C\overline{y}^*C^*\overline{y}\cr
&&-y^*\overline{C}y\overline{y}^*x'a' + y^*\overline{C}yx\overline{y}b'
-y^*\overline{C}yxx'\overline{C}^* + y^*\overline{C}y\overline{y}^*C^*\overline{y})
\end{eqnarray}
Adding these together, and collecting like terms we find the coefficient of $aa'$ is
\begin{eqnarray}
(xx')^2+y\overline{y}y^*\overline{y}^* + xx'y\overline{y}^*
+xx'\overline{y}y^*&=&
(xx'+y\overline{y}^*)(xx'+\overline{y}y^*)\cr
&=&1
\end{eqnarray}
and similarly, so is the coefficient of $bb'$.
Next, the coefficient of $ab'$ is
$$2x^2y\overline{y}-\Tr(x^2y\overline{y})=0,$$
while the coefficient of $a$ is
\begin{eqnarray}
&& -x^2x'\Tr(\overline{C}^*y) + xy\overline{y}\Tr(\overline{C}^*y^*)
+\Tr(x^2x'y\overline{C}^*)-\Tr(xy\overline{y}^*C^*\overline{y})\cr
&=& x^2x'\Tr(\overline{C}^*y-y\overline{C}^*)
+ xy\overline{y}\Tr(\overline{C}^*y^*-C^*\overline{y}^*)\cr
&=&0,
\end{eqnarray}
and the coefficient of $b$ is
$$-y^*\overline{y}^*x'\Tr(\overline{C}^*y) + xx'^2\Tr(\overline{C}^*y^*)
-xx'^2\Tr(y^*\overline{C}^*)+x'\Tr(y^*\overline{y}C^*\overline{y}^*)=0.$$
The remaining terms are as follows:
\begin{eqnarray}
&&(xx')^2\Tr(C\overline{C}^*) + \Tr(y^*\overline{C}y\overline{y}^*C^*\overline{y})\cr
&&+xx'(\Tr(C\overline{y}^*)\Tr(\overline{C}^*y) + \Tr(\overline{C}y)\Tr(C^*\overline{y}^*)
- 2\Tr(C^*\overline{y}C\overline{y}^*))
\end{eqnarray}
Using the lemma, we have
\begin{eqnarray}
\Tr(C\overline{y}^*)\overline{y}&=&\overline{y}C\overline{y}^* + \overline{y}y^*\overline{C}\cr
\Rightarrow \Tr(C\overline{y}^*)\Tr(\overline{y}C^*)&=&\Tr(\overline{y}C\overline{y}^*C^*)
+\overline{y}y^*\Tr(\overline{C}C^*)\cr
\Tr(\overline{C}y)y^*&=&y^*\overline{C}y+y^*\overline{y}C\cr
\Rightarrow \Tr(\overline{C}y)\Tr(y^*\overline{C}^*)&=&\Tr(y^*\overline{C}y\overline{C}^*)
+y^*\overline{y}\Tr(C\overline{C}^*)
\end{eqnarray}
and hence this expression
 reduces to $\Tr(C\overline{C}^*)$. 
This concludes the proof that the given matrices lie in the compact real form of $E_6$.

It is not hard to see that
the given generators $M$ satisfy the relation $$(\overline{M}')^\top M=I,$$
which may also be expressed by saying that they centralize the twisted
duality map induced by
$$M\mapsto ((\overline{M}')^\top)^{-1},$$
which is the product of complex conjugation with the ordinary duality map
whose centralizer is $F_4$.


\section{Aschbacher's construction of ${}^2E_6(q)$.}
Aschbacher defines ${}^2E_6(q)$ just in terms of Dickson's cubic form,
by defining the natural basis to be orthonormal for the Hermitian form.
He does not mention the Jordan algebra or octonions at all.
Of course this is equivalent to taking the space 
$\mathbb J$ over the field of order $q^2$, and
defining the Hermitian form so that the vectors
$(1,0,0\mid 0,0,0)$, $(0,0,0\mid e_i,0,0)$ and rotations form an orthonormal basis.

For fields of odd characteristic, we can just mimic everything we did for the
compact real form. However, in characteristic $2$ we have the usual problem that
the octonions are not spanned by the vectors $i_t$. It is necessary therefore
to change basis to the $e_i$ before reducing modulo $2$.
With this small change, we obtain generators for all the finite groups ${}^2E_6(q)$
in all characteristics.

\section{Another real form of $E_6$, and an alternative construction of ${}^2E_6(q)$.}

There is another Hermitian form one might want to use on the octonions, namely $h_1$
defined by
$$h_1(A)=\sum_i \beta_i\beta_{-i}' = A\overline{A}'+A'\overline{A}
=\overline{A}A'+\overline{A}'A.$$ 
This induces the Hermitian form $H_1$ on the Albert algebra, where
$$H_1(a,b,c\mid A,B,C) = aa'+bb'+cc'+h_1(A)+h_1(B)+h_1(C).$$
This Hermitian form is not positive definite, so it defines a non-compact real form of $E_6$,
in fact the form called $E_{6(2)}$. But on reducing modulo $p$, we again obtain
the finite groups ${}^2E_6(q)$. For certain purposes, this basis seems to be more useful
than the one Aschbacher uses. Since this construction does not (so far as I am aware) appear
explicitly in the literature, we give some more details here.

The group ${}^2E_6(q)$ is usually defined as the subgroup of $E_6(q^2)$
consisting of those elements which commute with the automorphism which
is the product of the automorphism given above with the field automorphism
$x\mapsto x^q$ on all coefficients.
To generate ${}^2E_6(q)$, therefore, 
we first need to take the Albert algebra over $\FF_{q^2}$.
 Let $\mathbb J=\mathbb J_F$,
where $F=\mathbb F_{q^2}$. Denote by $'$ the automorphism
of $\mathbb O_F$ induced by the Frobenius automorphism $\lambda\mapsto \lambda^q$
of $F$ of order $2$, that is, if $x=\sum_{i\in\pm I} \lambda_ie_i$ then
$x'=\sum_{i\in \pm I} \lambda_i^q e_i$.
Then there is a twisted duality map $*$ on $\mathbb J$ defined by
$$X^* = X'^\top=\overline{X}'.$$
This induces the above-mentioned automorphism of the group, which acts on
the generators $M$ by $M\mapsto (\overline{M}'^\top)^{-1}$.
For $M$ to centralize this automorphism, therefore, we must have
$\overline{M}'^\top M=I$.

There is a notion of \emph{twisted Jordan algebra}, in which there is a
new product $*$ defined in terms of the ordinary Jordan product $X\circ Y$
by $$X*Y = (X\circ Y)'.$$

Here we shall define the group in a slightly different way,
as hinted above. Let $H_1$ be the Hermitian form
defined on $\mathbb J_F$ by
$$H_1(a,b,c\mid A,B,C) = aa'+bb'+cc'+\Tr(A\overline{A}'+B\overline{B}'+C\overline{C}').$$
Then the simply-connected group ${}^2SE_6(q)$ is the subgroup of
$SE_6(q^2)$ which preserves $H_1$.
As long as the characteristic is not $2$, this Hermitian form may be described
in terms of the Jordan algebra as $H(X)=\Tr(X\circ X')$.  

In order to produce generators for ${}^2SE_6(q)$,
we consider matrices $M$ which satisfy $M^\dagger M=I$, where
$M^\dagger$ is defined by applying the field automorphism $x\mapsto x^q$ to
every coefficient in $\overline{M}^\top$. 
For example, if $x=\lambda e_i$ for some $i\in\pm I$, then the matrix
$$N_x=\begin{pmatrix}1&x&0\cr -\overline{x}' & 1& 0\cr 0&0&1\end{pmatrix}$$
is such a matrix. Since $e_i\overline{e_i}=0$ we have $x\overline{x}'=0$, so
$$\begin{pmatrix}1&x\cr -\overline{x}' & 1\end{pmatrix} =
\begin{pmatrix}1&x\cr 0&1\end{pmatrix}\begin{pmatrix}1&0\cr -\overline{x}'&1\end{pmatrix}$$
and therefore the given matrix $N_x$ lies in $SE_6(q^2)$. To check that it preserves $H_1$,
we first prove a small lemma.
\begin{lemma}
If $x\overline{x}=0$, and $y,z \in \mathbb O_F$, then
\begin{enumerate}
\item $x(yx)=x\Tr(yx)$;
\item $\Tr((xy)(z\overline{x}))=0$.
\end{enumerate}
\end{lemma}
\begin{proof}
\begin{enumerate}
\item $x(yx-\Tr(yx))=x(\overline{yx})=x(\overline{x}.\overline{y})=(x\overline{x})
\overline{y}=0$.
\item $\Tr((xy)(z\overline{x}))=\Tr((z\overline{x})(xy))=\Tr(((z\overline{x})x)y)
=\Tr((z(\overline{x}x))y)=0$.
\end{enumerate}
\end{proof}

Now by explicit computation we see that $N_x$ maps $X=(a,b,c\mid A,B,C)$ to
$$(a-\Tr(C\overline{x}'), b+\Tr(\overline{C}x), c\mid
A+\overline{x}\overline{B}, B-\overline{A}\overline{x}',
C+ax-bx'-x'\overline{C}x).$$
Then we can compute the new value of $H_1$ term by term as follows:
\begin{eqnarray}
aa'&\mapsto&aa'-a\Tr(C'\overline{x}) -a'\Tr(C\overline{x}')
+\Tr(C'\overline{x})\Tr(C\overline{x}')\cr
bb'&\mapsto& bb'+b\Tr(\overline{C}'x')+b'\Tr(\overline{C}x)
+\Tr(\overline{C}x)\Tr(\overline{C}'x')\cr
cc'&\mapsto& cc'\cr
\Tr(A\overline{A}')&\mapsto &
\Tr(A\overline{A}') + \Tr(\overline{A}'\overline{x}\overline{B}) + \Tr(AB'x')
+\Tr((\overline{x}\overline{B})(B'x'))\cr
\Tr(B\overline{B}')&\mapsto& \Tr(B\overline{B}')
-\Tr(\overline{A}\overline{x}'\overline{B}') - \Tr(BxA')
+\Tr((\overline{A}\overline{x}')(xA'))\cr
\Tr(C\overline{C}')&\mapsto& \Tr(C\overline{C}') + 
\Tr(a'C\overline{x}'+ax\overline{C}' -b'C\overline{x}-bx'\overline{C}'\cr&&\qquad-C(\overline{x}'C'\overline{x})-C'(\overline{x}C\overline{x}'))
\end{eqnarray}
and using the lemma
we see that all cross terms cancel out, as required.

We also see that $F_4(q)$ is a subgroup of ${}^2SE_6(q)$, because on the 
$\mathbb F_q$-subspace $\mathbb J_{\mathbb F_q}$ all elements of $F_4(q)$
preserve the standard norm, which is just the restriction of $H_1$.
We may now take the same generators for  $F_4(q)$ as before, consisting of
certain matrices $M$ which are fixed by the field automorphism.

Then adjoin to $F_4(q)$ the matrix 
$$M=\begin{pmatrix}x&0&0\cr 0&x^q&0\cr 0&0&1\end{pmatrix},$$
where $x\in \FF_{q^2}\setminus \FF_q$ satisifies $x^{1+q}=1$.
More generally, take matrices
$$M=\begin{pmatrix}a&b&0\cr -b^q&a^q&0\cr
0&0&1\end{pmatrix},$$
where $a^{1+q}+b^{1+q}=1$.

The extra root elements are given by matrices like
$$
\begin{pmatrix}1&\lambda e_0&0\cr -\lambda^q e_{-0}&1&0\cr 0&0&1\end{pmatrix},\mbox{ and }
\begin{pmatrix}1&\lambda e_i&0\cr \lambda^q e_i&1&0\cr 0&0&1\end{pmatrix}, \mbox{ for }i=\pm1,\pm\omega,\pm\ombar.$$



With a certain amount of calculation it is now possible to show that this group
has exactly three orbits on the white points for $E_6(q^2)$. The lengths of these
orbits are as follows:
\begin{enumerate}
\item $(q^9+1)(q^{12}-1)(q^5+1)/(q^2-1)$, 
\item $(q^4+1)(q^9+1)q^5(q^{12}-1)(q^3-1)/(q^2-1)$ and
\item $q^{16}(q^8+q^4+1)(q^9+1)/(q+1)$. 
\end{enumerate}
Of these, the first two are isotropic with respect to $H_1$, while the last is
non-isotropic.
Now we know the stabilizer in $SE_6(q^2)$ has shape $q^{32}.\Spin^+_{10}(q^2).C_{q^2-1}$
and it is now not too difficult to see what the stabilizers in ${}^2SE_6(q)$ must be.
A point in the last orbit has a stabilizer of shape $\Spin_{10}^-(q).C_{q+1}$, from which
we deduce the order of ${}^2SE_6(q)$, that is,
$$|^2SE_6(q)|= q^{36}(q^{12}-1)(q^9+1)(q^8-1)(q^6-1)(q^5+1)(q^2-1).$$

The three orbits are distinguished as follows. Any white vector $v$ determines a $17$-space,
which is the radical of the quadratic form determined by $v$, and hence
determines the radical of $H_1$ on this $17$-space. If $v$ belongs to this last space,
then $v$ is of type (1), which Aschbacher calls \emph{emerald}; and in fact
the radical of $H_1$ on the $17$-space is just $\langle v\rangle$. Otherwise,
if $H_1(v)=0$, then $v$ is of type (2). Finally, $v$ is of type (3) if $H_1(v)\ne 0$.


\end{document}